\documentclass[11pt,english]{amsart}
\usepackage[T1]{fontenc}
\usepackage{geometry}
\usepackage{amsthm}
\usepackage{amssymb}

\providecommand{\tabularnewline}{\\}

\numberwithin{equation}{section} 
\numberwithin{figure}{section} 
\theoremstyle{plain}
\newtheorem{thm}{Theorem}[section]
  \theoremstyle{plain}
  \newtheorem{prop}[thm]{Proposition}
  \theoremstyle{definition}
  \newtheorem{defn}[thm]{Definition}
  \theoremstyle{remark}
  \newtheorem{note}[thm]{Note}
 \theoremstyle{definition}
  \newtheorem{example}[thm]{Example}
  \theoremstyle{plain}
  \newtheorem{lem}[thm]{Lemma}
  \theoremstyle{plain}
  \newtheorem{algorithm}[thm]{Algorithm}
  \theoremstyle{plain}
  \newtheorem{cor}[thm]{Corollary}

\usepackage[all]{xy}
\usepackage{stmaryrd}
\usepackage[algo2e,figure,vlined]{algorithm2e}
\usepackage[dvips]{graphicx}
\usepackage{multicol}

\AtBeginDocument{
  
}

\usepackage{babel}

\begin{document}

\title{a graph pebbling algorithm on weighted graphs}

\author{N\'Andor Sieben}
\begin{abstract}
A pebbling move on a weighted graph removes some pebbles at a vertex
and adds one pebble at an adjacent vertex. The number of pebbles removed
is the weight of the edge connecting the vertices. A vertex is reachable
from a pebble distribution if it is possible to move a pebble to that
vertex using pebbling moves. The pebbling number of a weighted graph
is the smallest number $m$ needed to guarantee that any vertex is
reachable from any pebble distribution of $m$ pebbles. Regular pebbling
problems on unweighted graphs are special cases when the weight on
every edge is 2. A regular pebbling problem often simplifies to a
pebbling problem on a simpler weighted graph. We present an algorithm
to find the pebbling number of weighted graphs. We use this algorithm
together with graph simplifications to find the regular pebbling number
of all connected graphs with at most nine vertices.
\end{abstract}

\address{Northern Arizona University, Department of Mathematics and Statistics,
Flagstaff AZ 86011-5717, USA}

\email{nandor.sieben@nau.edu}

\keywords{weighted graph pebbling}

\subjclass[2000]{05C99}

\date{\the\month/\the\day/\the\year}

\maketitle
\global\long\def\a{\hbox{$\shortrightarrow$}}

 \setcounter{topnumber}{4} \setcounter{totalnumber}{4}

\section{Introduction}

Graph pebbling has its origin in number theory. It is a model for
the transportation of resources. Starting with a pebble distribution
on the vertices of a simple connected graph, a \emph{pebbling move}
removes two pebbles from a vertex and adds one pebble at an adjacent
vertex. We can think of the pebbles as fuel containers. Then the loss
of the pebble during a move is the cost of transportation. A vertex
is called \emph{reachable} if a pebble can be moved to that vertex
using pebbling moves. The \emph{pebbling number} of a graph is the
minimum number of pebbles that guarantees that every vertex is reachable.
There are many different variations of pebbling. For a comprehensive
list of references for the extensive literature see the survey papers
\cite{Hurlbert_survey1,Hurlbert_survey2}.

Our goal is to find an algorithm that finds the pebbling number in
a realistic amount of computing time. The main idea of the algorithm
is that if we know all the \emph{sufficient} distributions from which
a given goal vertex is reachable then we can find the \emph{insufficient}
distributions from which the goal vertex is not reachable. An insufficient
distribution must be smaller than a sufficient distribution. The pebbling
number can be found by finding an insufficient distribution with the
most pebbles. The problem is that there are too many sufficient distributions.
Luckily it suffices to find the \emph{barely sufficient} distributions
from which the goal vertex is no longer reachable after the removal
of any pebble. 

Our algorithm works even if the cost of moving a pebble from one vertex
to another varies between different vertices. To take advantage of
this, we generalize the notion of graph pebbling on weighted graphs
and we develop the basic theory of weighted graph pebbling.

The generalization is worth the effort since pebbling on many graphs
can be simplified if we replace the graph by a weighted graph with
fewer edges. For example a tree can be replaced by a weighted graph
containing a single edge. Cut vertices, leaves and ears offer the
most fruitful simplifications.

We use these simplifications and our algorithm to calculate the pebbling
number of all connected graphs with at most nine vertices. We present
the spectrum of pebbling numbers in terms of the number of vertices
in the graph.

\section{Preliminaries}

Let $G$ be a simple connected graph. We use the notation $V(G)$
for the vertex set and $E(G)$ for the edge set. We use the standard
notation $vu=uv$ for the edge $\{v,u\}\in E(G)$. A \emph{path} of
$G$ is a subgraph isomorphic to the path graph $P_{n}$ with $n\ge1$
vertices. A \emph{weighted graph} $G_{\omega}$ is a graph $G$ with
a weight function $\omega:E(G)\to{\bf N}$. 

A \emph{pebble function} on $G$ is a function $p:V(G)\to{\bf Z}$
where $p(v)$ is the number of pebbles placed at $v$. A \emph{pebble
distribution} is a nonnegative pebble function. The \emph{size} of
a pebble distribution $p$ is the total number of pebbles $\|p\|=\sum_{v\in V(G)}p(v)$.
The \emph{support} of the pebble distribution $p$ is the set $\text{supp}(p)=\{v\in V(G)\mid p(v)>0\}$.
We are going to use the notation $p(v_{1},\ldots,v_{n},*)=(a_{1},\ldots,a_{n},q(*))$
to indicate that $p(v_{i})=a_{i}$ for $i\in\{1,\ldots,n\}$ and $p(w)=q(w)$
for all $w\in V(G)\setminus\{v_{1},\ldots,v_{n}\}$.

If $vu\in E(G)$ then the \emph{pebbling move} $(v\a u)$ on the weighted
graph $G_{\omega}$ removes $\omega(vu)$ pebbles at vertex $v$ and
adds one pebble at vertex $u$, more precisely, it replaces the pebble
function $p$ with the pebble function\[
p_{(v\a u)}(v,u,*)=(p(v)-\omega(vu),p(u)+1,p(*)).\]
 Note that the resulting pebble function $p_{(v\a u)}$ might not
be a pebble distribution even if $p$ is. 

The \emph{inverse} of the pebbling move $(v\a u)$ is denoted by $(v\a u)^{-1}$.
The inverse removes a pebble from $u$ and adds two pebbles at $v$,
that is, it creates the new distribution $p_{(v\a u)^{-1}}(v,u,*)=(p(v)+2,p(u)-1,p(*))$.
Note that $(v\a u)^{-1}$ is not a pebbling move.

A \emph{pebbling sequence} is a finite sequence $s=(s_{1},\ldots,s_{k})$
of pebbling moves. The pebble function gotten from the pebble function
$p$ after applying the moves in $s$ is denoted by $p_{s}$. The
concatenation of the pebbling sequences $r=(r_{1},\ldots,r_{k})$
and $s=(s_{1},\ldots,s_{l})$ is denoted by $rs=(r_{1},\ldots,r_{k},s_{1},\ldots,s_{l})$.

A pebbling sequence $(s_{1},\ldots,s_{n})$ is \emph{executable} from
the pebble distribution $p$ if $p_{(s_{1},\ldots,s_{i})}$ is nonnegative
for all $i\in\{1,\ldots,n\}$. A vertex $x$ of $G$ is \emph{$t$-reachable}
from the pebble distribution $p$ if there is an executable pebbling
sequence $s$ such that $p_{s}(x)\ge1$. We say $x$ is \emph{reachable}
if it is $1$-reachable. 

We write $\pi_{t}(G_{\omega},x)$ for the minimum number $m$ such
that $x$ is $t$-reachable from every pebble distribution of size
$m$. We use the notation $\pi(G_{\omega},x)$ for $\pi_{1}(G_{\omega},x)$.
The \emph{$t$-pebbling number} $\pi_{t}(G_{\omega})$ is $\max\{\pi_{t}(G_{\omega},x)\mid x\in V(G)\}$.
The \emph{pebbling number} $\pi(G_{\omega})$ is the 1-pebbling number
$\pi_{1}(G_{\omega})$.

If $\omega(e)=2$ for all $e\in E(G)$ then $\pi(G_{\omega})=\pi(G)$
is the usual unweighted pebbling number. So we allow the weight function
$\omega$ to be defined only on a subset of $V(G)$ and use the default
weight of 2 for edges where $\omega$ is undefined. 

Changing the order of moves in an executable pebbling sequence $s$
may result in a sequence $r$ that is no longer executable. On the
other hand the ordering of the moves has no effect on the resulting
pebble function, that is, $p_{s}=p_{r}$. This motivates the following
definition. 

Given a multiset $S$ of pebbling moves on the weighted graph $(G_{\omega})$,
the \emph{transition digraph} $T(G,S)$ is a directed multigraph whose
vertex set is $V(G)$, and each move $(v\a u)$ in $S$ is represented
by a directed edge $(v,u)$. The transition digraph of a pebbling
sequence $s=(s_{1},\ldots,s_{n})$ is $T(G,s)=T(G,S)$, where $S=\{s_{1},\ldots,s_{n}\}$
is the multiset of moves in $s$. Let $d_{T(G,S)}^{-}$ denote the
in-degree and $d_{T(G,S)}^{+}$ the out-degree in $T(G,S)$. We simply
write $d^{-}$ and $d^{+}$ if the transition digraph is clear from
context. It is easy to see that the pebble function gotten from $p$
after applying the moves in a multiset $S$ of pebbling moves in any
order satisfies\[
p_{S}(v)=p(v)+d_{T(G,S)}^{-}(v)-\sum\{\omega(vu)\mid(v,u)\in E(T(G,S))\}\]
for all $v\in G$. For unweighted graphs the formula simplifies to\[
p_{S}(v)=p(v)+d_{T(G,S)}^{-}(v)-2d_{T(G,S)}^{+}(v).\]

\section{Cycles in the transition digraph}

In this section we present a version of the No-Cycle Lemma \cite{Betsy,Milans,Moews}.
If the pebbling sequence $s$ is executable from a pebble distribution
$p$ then we clearly must have $p_{s}\ge0$. We say that a multiset
$S$ of pebbling moves is \emph{balanced} with a pebble distribution
$p$ \emph{at vertex} $v$ if $p_{S}(v)\ge0$. The multiset $S$ is
\emph{balanced} with $p$ if $S$ is balanced with $p$ at all $v\in V(G)$,
that is, $p_{S}\ge0$. We say that a pebbling sequence $s$ is \emph{balanced}
with $p$ if the multiset of moves in $s$ is balanced with $p$.
The balance condition is necessary but not sufficient for a pebbling
sequence to be executable. A multiset of pebbling moves or a pebbling
sequence is called \emph{acyclic} if the corresponding transition
digraph has no directed cycles. 
\begin{prop}
\label{pro:untangling}If $S$ is a multiset of pebbling moves on
$G_{\omega}$ then there is an acyclic multiset $R\subseteq S$ such
that $p_{R}\ge p_{S}$ for all pebble function $p$ on $G$.\end{prop}
\begin{proof}
Let $p$ be a pebble function on $G$. Suppose that $T(G,S)$ has
a directed cycle $C$. Let $Q$ be the multiset of pebbling moves
corresponding to the arrows of $C$ and $R=S\setminus Q$. Let $u_{v}$
be the first vertex from $v$ along $C$. Then $p_{R}(v)=p_{S}(v)-1+\omega(vu_{v})\ge p_{S}(v)$
for $v\in V(C)$ and $p_{R}(v)=p_{S}(v)$ for $v\in V(G)\setminus V(C)$. 

We can repeat this process until we eliminate all the cycles. We finish
in finitely many steps since every step decreases the number of pebbling
moves.\end{proof}
\begin{defn}
Let $S$ be a multiset of pebbling moves on $G$. An element $(v\a u)\in S$
is called an \emph{initial move} of $S$ if $d^{-}(v)=0$. A pebbling
sequence $s$ is called \emph{regular} if $s_{i}$ is an initial move
of $S\setminus\{s_{1},\ldots,s_{i-1}\}$ for all $i$.
\end{defn}
It is clear that if the multiset $S$ is balanced with a pebble distribution
$p$ and $s$ is an initial move of $S$ then $s$ is executable from
$p$.
\begin{prop}
\label{pro:regular}If $S$ is an acyclic multiset then there is a
regular sequence $s$ of the elements of $S$. If $S$ is also balanced
with the pebble function $p$ then $s$ is executable from $p$. \end{prop}
\begin{proof}
If $S$ is acyclic then we must have an initial move $t$ of $S$.
Then $S\setminus\{t\}$ is still acyclic. So we can recursively find
the elements of $s$ recursively by picking an initial move $t$ of
$S$ and then replacing $S$ with $S\setminus\{t\}$ at each step.

Now assume that $S$ is balanced with $p$. Then $S_{i}=S\setminus\{s_{1},\ldots,s_{i-1}\}$
is balanced with $p_{(s_{1},\ldots,s_{i-1})}$ for all $i$ since
$(p_{(s_{1},\ldots,s_{i-1})}){}_{S_{i}}=p_{S}\ge0$. Hence the initial
move $s_{i}$ of $S_{i}$ is executable from $p_{(s_{1},\ldots,s_{i-1})}$,
that is, $p_{(s_{1},\ldots,s_{i})}\ge0$ for all $i$. 
\end{proof}
The following result is our main tool.
\begin{thm}
\label{thm:fundam}Let $p$ be a pebble distribution on $G_{\omega}$
and $x\in V(G)$. The following are equivalent.
\begin{enumerate}
\item Vertex $x$ is reachable from $p$. 
\item There is a multiset $S$ of pebbling moves with $p_{S}\ge0$ and $p_{S}(x)\ge1$.
\item There is an acyclic multiset $R$ of pebbling moves with $p_{R}\ge0$
and $p_{R}(x)\ge1$.
\item Vertex $x$ is reachable from $p$ through a regular pebbling sequence.
\end{enumerate}
\end{thm}
\begin{proof}
If $x$ is reachable from $p$ then there is a sequence $s$ of pebbling
moves such that $s$ is executable from $p$ and $p_{s}(x)\ge1$.
If $S$ is the multiset of the moves of $s$ then $p_{S}\ge0$ and
$p_{S}(x)\ge1$ and so (1) implies (2).

By Proposition~\ref{pro:untangling}, (2) implies (3) and by Proposition~\ref{pro:regular},
(3) implies (4). It is clear that (4) implies (1). 
\end{proof}
It is convenient to write the condition $p_{S}\ge0$ and $p_{S}(x)\ge1$
compactly as $p_{S}\ge1_{\{x\}}$ using the indicator function of
the singleton set $\{x\}$.

\section{Cut vertices}

The pebbling number of a graph with a cut vertex often can be calculated
using a simpler graph. This simplification introduces new weights.
The following theorem is the main reason we study weighted graphs. 
\begin{prop}
\label{pro:cut}Let $H$ and $K$ be connected graphs such that $V(H)\cap V(K)=\{v\}$
and $v$ is a cut vertex of $G=H\cup K$. Let $\omega$ be a weight
function on $E(G)$. Assume that $\pi_{t}(K_{\omega},v)=at+b$ for
all $t$. Define a graph $\tilde{G}$ by $V(\tilde{G})=V(H)\dot{\cup}\{u\}$
and $E(\tilde{G})=E(H)\cup\{vu\}$. Define a weight function on $E(\tilde{G})$
by \[
\tilde{\omega}(e)=\begin{cases}
a & \text{if }e=vu\\
\omega(e) & \text{else}\end{cases}.\]
If the goal vertex $x$ is in $V(H)$ then $\pi(G_{\omega},x)=\pi(\tilde{G}_{\tilde{\omega}},x)+b$.
\end{prop}
\begin{figure}
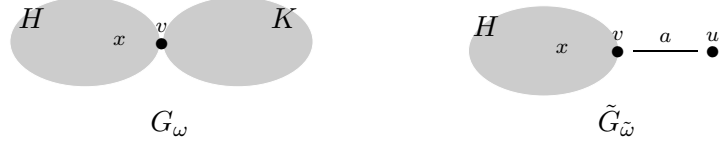

\begin{tabular}{cccccccccccc}
\input{cut1.tex} &  &  &  &  &  &  &  &  &  &  & \input{cut2.tex}\tabularnewline
$G_{\omega}$ &  &  &  &  &  &  &  &  &  &  & $\tilde{G}_{\tilde{\omega}}$\tabularnewline
\end{tabular}

\caption{\label{cap:cut}Simplification using the cut vertex $v$. If $\pi_{t}(K_{\omega},v)=at+b$
then $\pi(G_{\omega},x)=\pi(\tilde{G}_{\tilde{\omega}},x)+b$.}

\end{figure}

To simplify notation, we used $K_{\omega}$ instead of the more precise
$K_{\omega|E(K)}$ even though $\omega$ is defined on values outside
of $E(K)$. 
\begin{proof}
The graphs are visualized in Figure~\ref{cap:cut}. First we show
that $\pi(G_{\omega},x)\ge\pi(\tilde{G}_{\tilde{\omega}},x)+b$. Let
$p$ be a pebble distribution on $\tilde{G}$ with $\|p\|=\pi(G_{\omega},x)-b$.
We create a new distribution $q$ on $G$ consisting of red and green
pebbles. The red pebbles are placed on $H$ exactly the same way as
the pebbles in $p$ are placed on $H$. The number of green pebbles
is $p(u)+b$. The green pebbles are placed on $K$ so that the number
of pebbles that can be moved to $v$ using only green pebbles is minimum.
This minimum number is clearly $\lfloor p(u)/a\rfloor$. Note that
$q$ can have both red and green pebbles on $v$. Then $\|q\|=\pi(G_{\omega},x)$,
so there is an acyclic multiset $S$ of pebbling moves on $G_{\omega}$
such that $q_{S}\ge1_{\{x\}}$. Let $S_{H}$ and $S_{K}$ contain
the moves of $S$ inside $H$ and $K$ respectively so that $S=S_{H}\dot{\cup}S_{K}$. 

We are going to see that moving red pebbles from $H$ to $K$ is not
beneficial and so these moves can be eliminated. Since $S$ is acyclic,
any maximal walk in $T(K,S_{K})$ starting at $v$ is actually a path.
Let us remove the pebbling moves corresponding to such maximal walks
from $S_{K}$ until we eliminate all walks from $T(K,S_{K})$ starting
at $v$. The choice of these maximal walks is not unique and they
can overlap, so we need to eliminate them one by one. The resulting
multiset $S'_{K}\subseteq S_{K}$ is balanced with $q$ and $q_{S_{K}}(v)\le q_{S'_{K}}(v)$. 

Executing $S'_{K}$ from $q$ cannot move more than $\lfloor p(u)/a\rfloor$
pebbles to $v$ since $S'_{K}$ does not have any effect on the red
pebbles. Let $R$ be the multiset containing the elements of $S_{H}$
together with $\lfloor\|p(u)\|/a\rfloor$ copies of the move $(u\a v)$.
Then it is clear that $p_{R}(u)\ge0$ and $p_{R}$ is not smaller
than $q_{S}$ on $H$, hence $p_{R}\ge1_{\{x\}}$.

Now we show that $\pi(G_{\omega},x)\le\pi(\tilde{G}_{\tilde{\omega}},x)+b$.
Let $p$ be a pebble distribution on $G$ with $\|p\|=\pi(\tilde{G}_{\tilde{\omega}},x)+b$.
Let $c=\|p|_{V(K)\setminus\{v\}}\|$ be the number of pebbles in this
distribution on $V(K)\setminus\{v\}$. We create a new distribution
$q$ on $\tilde{G}$ such that $q$ and $p$ are the same on $V(H)$
and $q(u)=\max\{0,c-b\}$. Then $\|q\|\ge\pi(\tilde{G}_{\tilde{\omega}},x)$
so there is an acyclic multiset $S$ of pebbling moves on $\tilde{G}_{\tilde{\omega}}$
such that $q_{S}\ge1_{\{x\}}$. There is a multiset $R_{1}$ of pebbling
moves on $K$ using only the pebbles on $V(K)\setminus\{v\}$ such
that $p_{R_{1}}(v)=p(v)+\lfloor\max\{0,c-b\}/a\rfloor$. Let $R$
be the multiset containing the elements of $R_{1}$ together with
the elements of $S$ different from $(u\a v)$. Then $p_{R}$ is not
smaller than $q_{S}$ on $H$ and $p_{R}$ is nonnegative on $V(K)\setminus\{v\}$,
hence $p_{R}\ge1_{\{x\}}$.\end{proof}
\begin{note}
\label{not:t-results}The previous proposition is applicable in many
situations since the function $t\mapsto\pi_{t}(G,x)$ is often linear;
for example for trees, complete graphs and hypercubes. In particular
it is linear for cycles \cite{tpebbling} where $\pi_{t}(C_{2n})=t2^{n}$
and $\pi_{t}(C_{2n+1})=1+(t-1)2^{n}+2\left\lfloor \frac{2^{n+1}}{3}\right\rfloor $. 

The simplest nonlinear example is the wheel graph $W_{5}$ with 5
vertices. If $x$ is a degree 3 vertex then\[
\pi_{t}(W_{5},x)=\begin{cases}
5 & \text{if }t=1\\
4t & \text{if }t\ge2\end{cases}.\]
If $G$ is the complete graph with 7 vertices with one missing edge
$xy$ then \[
\pi_{t}(G,x)=\begin{cases}
2t+5 & \text{if }t\in\{1,2\}\\
4t & \text{if }t\ge3\end{cases}.\]

\end{note}

\section{\label{sec:Simp-leaves}Simplifications using leaves}
\begin{prop}
\label{pro:leaves1}Let $G_{\omega}$ be a weighted star graph with
center $x$ and spikes $xv_{1},\ldots,xv_{n}$. Suppose that $a=\omega(xv_{1})$
is the maximum value of $\omega$. Then $\pi_{t}(G_{\omega},x)=ta+\sum_{i=2}^{n}(\omega(xv_{i})-1)$.\end{prop}
\begin{proof}
The maximum number of pebbles we can place on $v_{i}$ so that at
most $t_{i}$ pebbles can be moved from $v_{i}$ to $x$ is $(t_{i}+1)\omega(xv_{i})-1$.
So\begin{align*}
\pi_{t}(G_{\omega},x) & =\max\{\sum_{i=1}^{n}((t_{i}+1)\omega(xv_{i})-1)\mid t_{1}+\cdots+t_{n}<t\}+1\\
 & =(t-1+1)\omega(xv_{1})-1+\sum_{i=2}^{n}((0+1)\omega(xv_{i})-1)+1\\
 & =ta+\sum_{i=2}^{n}(\omega(xv_{i})-1)\end{align*}
since the maximum is taken when $t_{1}=t-1$ and $t_{2}=\cdots=t_{n}=0$.
\end{proof}
The reader can easily verify the following result.
\begin{prop}
\label{pro:leaves2}Let $x$, $v_{1}$ and $v_{2}$ be the consecutive
vertices of the graph $G=P_{3}$ with weight function $\omega$. Then
$\pi_{t}(G_{\omega},x)=t\omega(xv_{1})\omega(v_{1}v_{2})$.
\end{prop}
\begin{figure}
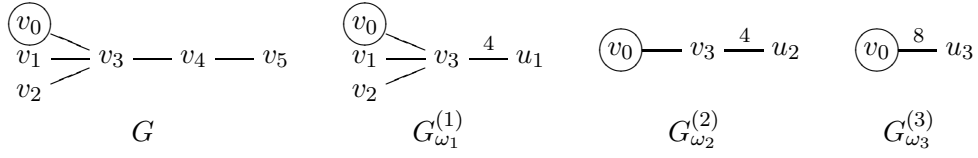

\begin{tabular}{cccc}
\input{leavesEx1.tex} & \input{leavesEx2.tex} & \input{leavesEx3.tex} & \input{leavesEx4.tex}\tabularnewline
$G$ & $G_{\omega_{1}}^{(1)}$ & $G_{\omega_{2}}^{(2)}$ & $G_{\omega_{3}}^{(3)}$\tabularnewline
\end{tabular}

\caption{\label{cap:leaves_ex}Simplification of a tree: $\pi(G,v_{0})=\pi(G_{\omega_{1}}^{(1)},v_{0})=\pi(G_{\omega_{2}}^{(2)},v_{0})+2=\pi(G_{\omega_{3}}^{(3)},v_{0})+2=10$.}

\end{figure}

Propositions~\ref{pro:cut}, \ref{pro:leaves1} and \ref{pro:leaves2}
allow us to calculate the pebbling number of every weighted tree since
we can simplify the tree to a single edge. The process is shown in
the next example.
\begin{example}
Figure~\ref{cap:leaves_ex} shows the stages of the simplification
of a tree. First we let $K$ be the subgraph of $G$ generated by
$\{v_{3},v_{4},v_{5}\}$. Then $\pi_{t}(K,v_{3})=4t$ by Proposition~\ref{pro:leaves2}
so we replace $K$ by the weighted edge $v_{3}u_{1}$ to get $G_{\omega_{1}}^{(1)}$.
Next we let $K$ be the subgraph of $G_{\omega_{1}}^{(1)}$ generated
by $\{v_{1},v_{2},u_{1}\}$. Then $\pi_{t}(K,v_{3})=4t+2$ by Proposition~\ref{pro:leaves1}
so we replace $K$ by the weighted edge $v_{3}u_{2}$ to get $G_{\omega_{2}}^{(2)}$.
Finally we use Proposition~\ref{pro:leaves2} again to get $G_{\omega_{3}}^{(3)}$.
\end{example}

\section{Simplification using ears}

In this section we use the existence of special paths in our graph
to simplify the calculation of the pebbling number. A \emph{thread}
of a graph is a path containing vertices of degree 2.
\begin{defn}
Let $x$ be a goal vertex in $G$. Let $v_{1},\ldots,v_{n}$ be the
consecutive vertices of a maximal thread $T$ not containing $x$.
There are unique vertices $v_{0}$ and $v_{n+1}$ outside of $T$
that are adjacent to $v_{1}$ and $v_{n}$ respectively. The subgraph
$E$ induced by $v_{0},\ldots,v_{n+1}$ is called an \emph{ear}. The
vertices of $T$ are called the \emph{inner vertices} of $E$. If
$v_{0}=v_{n+1}$ then $E$ is called a \emph{closed ear}. If the vertices
of $T$ are cut vertices then $E$ is called a \emph{cut ear}. If
$E$ is neither a closed ear nor a cut ear then it is called an \emph{open
ear}.
\end{defn}
Note that an ear has at least two edges. Also note that the goal vertex
can be an end vertex of an ear.

\subsection{\label{sub:Closed-ears}Closed ears}

If a closed ear has default weights then it can be replaced by a weighted
edge using Corollary~\ref{pro:cut} and Note~\ref{not:t-results}.
The simplification is shown in Figure~\ref{cap:closedear}. If the
closed ear has $2n$ vertices then $\pi(G_{\omega},x)=\pi(\tilde{G}_{\tilde{\omega}},x)$.
If the closed ear has $2n+1$ vertices then $\pi(G_{\omega},x)=\pi(\tilde{G}_{\tilde{\omega}},x)+1-2^{n}+2\left\lfloor \frac{2^{n+1}}{3}\right\rfloor $.
The edge weight is $a=2^{n}$ in both cases. 

\begin{figure}
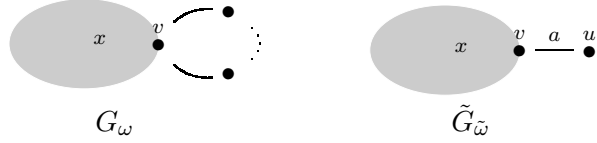

\begin{tabular}{ccccccccccccc}
\input{ear1.tex} &  &  &  &  &  &  &  &  &  &  &  & \input{ear1a.tex}\tabularnewline
$G_{\omega}$ &  &  &  &  &  &  &  &  &  &  &  & $\tilde{G}_{\tilde{\omega}}$\tabularnewline
\end{tabular}

\caption{\label{cap:closedear}Substitution for a closed ear. The edge weight
$a$ is $2^{\lfloor\frac{k}{2}\rfloor}$ where $k$ is the number
of vertices of the closed ear.}

\end{figure}

\subsection{\label{sub:Cut-ears}Cut ears}

Cut ears can be replaced by weighted edges as well. First we need
the following result. 
\begin{lem}
\label{lem:noinner}Let $p$ be a maximum size pebble distribution
from which the goal vertex $x$ is not reachable. Then $p$ has no
pebbles on the inner vertices of a cut ear.\end{lem}
\begin{proof}
Suppose $u$ is an inner vertex of the cut ear $E$ and $p(u)>0$.
Let $H$ and $K$ be the connected components of $G\setminus\{u\}$
such that $x\in H$. There is a unique vertex $v\in K$ that is adjacent
to $u$. The size of $q=p_{(v\a u)^{-1}}$ is larger than the size
of $p$. We show that $x$ is not reachable from $q$ which is a contradiction. 

Suppose there is an acyclic multiset $S$ of pebbling moves with $q_{S}\ge1_{\{x\}}$.
If $(v\a u)\in S$ then with $R=S\setminus\{(v\a u)\}$ we have $p_{R}=q_{S}\ge1_{\{x\}}$
which is not possible. So we can assume that $(v\a u)\not\in S$.
Let $R$ contain those moves of $S$ that do not involve any vertex
in $K$. Then $p_{R}(u)\ge q_{S}(u)+1$, $p_{R}(w)=q_{S}(w)$ for
$w\in V(H)$ and $p_{R}(w)=p(w)$ for $w\in V(K)$. So $p_{R}\ge1_{\{x\}}$
which is again impossible.
\end{proof}
\begin{figure}
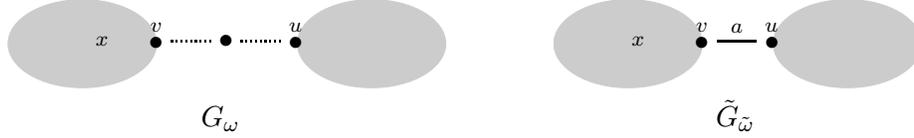

\begin{tabular}{ccccccccccccccccccc}
\input{ear2.tex} &  &  &  &  &  &  &  &  &  &  &  &  &  &  &  &  &  & \input{ear2a.tex}\tabularnewline
$G_{\omega}$ &  &  &  &  &  &  &  &  &  &  &  &  &  &  &  &  &  & $\tilde{G}_{\tilde{\omega}}$\tabularnewline
\end{tabular}

\caption{\label{cap:cutear}Substitution for a cut ear. The edge weight is
$\tilde{\omega}(vu)=a=2^{n-1}$ where $n$ is the number vertices
of the path connecting $v$ to $u$ in $G$. }

\end{figure}

\begin{prop}
\label{pro:cut-ear}Let $E$ be a cut ear of $G_{\omega}$ with end
vertices $v$ and $u$. Let $\tilde{G}$ be the graph created from
$G$ by removing the inner vertices of $E$ and adding the edge $vu$.
Define $\tilde{\omega}$ on $E(\tilde{G})$ by \[
\tilde{\omega}(e)=\begin{cases}
a & \text{if }e=vu\\
\omega(e) & \text{else}\end{cases}\]
where $a$ is the product of the weights of the edges of $E$. If
the goal vertex $x$ is not an inner vertex of $E$ then $\pi(G_{\omega},x)=\pi(\tilde{G}_{\tilde{\omega}},x)$.\end{prop}
\begin{proof}
Without loss of generality we can assume that $x$ is closer to $v$
than to $u$ as shown in Figure~\ref{cap:cutear}. Let $u=v_{1},v_{2},\ldots,v_{k}=v$
be the consecutive vertices of $E$.

First we show that $\pi(G_{\omega},x)\le\pi(\tilde{G}_{\tilde{\omega}},x)$.
For a contradiction, assume that $\pi(G_{\omega},x)>\pi(\tilde{G}_{\tilde{\omega}},x)$.
Let $p$ be a maximum size pebble distribution on $G$ from which
$x$ is not reachable. By Lemma~\ref{lem:noinner}, $p$ has no pebbles
on the inner vertices of $E$ so the restriction $q=p|_{V(\tilde{G})}$
is a pebble distribution on $\tilde{G}$ with $\|q\|=\|p\|=\pi(G_{\omega},x)-1\ge\pi(\tilde{G}_{\tilde{\omega}},x)$.
Hence there is a multiset $S$ of pebbling moves on $\tilde{G}_{\tilde{\omega}}$
such that $q_{S}\ge1_{\{x\}}$. Let $R$ be the multiset of pebbling
moves containing the moves in $S$ with each move of the form $(u\a v)$
replaced by the moves $(v_{1}\a v_{2}),\ldots,(v_{k-1}\a v_{k})$.
Then $q_{R}\ge1_{\{x\}}$ which is a contradiction.

Now we show that $\pi(G_{\omega},x)\ge\pi(\tilde{G}_{\tilde{\omega}},x)$.
Let $p$ be a pebble distribution on $\tilde{G}$ with size $\pi(G_{\omega},x)$.
Let $q$ be the extension of $p$ to $V(G)$ such that $q$ is zero
on the inner vertices of $E$. Then $\|q\|=\|p\|=\pi(G_{\omega},x)$
and so there is an acyclic multiset $S$ of pebbling moves such that
$q_{S}\ge1_{\{x\}}$. We create a multiset $R$ of pebbling moves
on $\tilde{G}$ as follows. We start with $S$. We then search for
a directed path in $T(G,S)$ connecting $u$ to $v$ and we remove
all the moves corresponding to the arrows of this directed path. We
do this until there are no more such directed paths. Then we add as
many copies of $(u\a v)$ as the number of directed paths removed.
Finally we remove all moves involving inner vertices of $E$. It is
easy to see that $p_{R}\ge1_{\{x\}}$.
\end{proof}

\subsection{Open ears}

\begin{figure}
\begin{tabular}{c}
\input{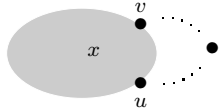}\tabularnewline
\end{tabular}

\caption{\label{cap:openear}A graph with an open ear.}

\end{figure}

Figure~\ref{cap:openear} depicts an open ear. An open ear cannot
be replaced by a single edge but we can still take advantage of it
using squishing as explained in Section~\ref{sec:Squished}.

\subsection{Examples}

\begin{figure}
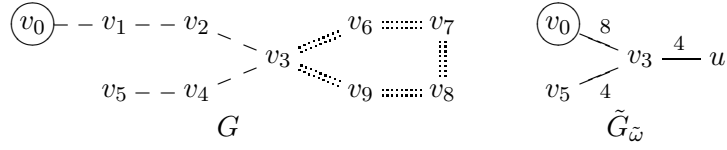

\begin{tabular}{ccc}
\input{simplifyExa.tex} &  & \input{simplifyEx1.tex}\tabularnewline
$G$ &  & $\tilde{G}_{\tilde{\omega}}$\tabularnewline
\end{tabular}

\caption{\label{cap:Simpli1}Simplification of $G$ with goal vertex $v_{0}$
such that $\pi(G,v_{0})=\pi(\tilde{G}_{\tilde{\omega}},v_{0})+1=36$.
The two cut ears denoted by dashed edges are replaced by the weighted
edges $v_{0}v_{3}$ and $v_{5}v_{3}$ in $\tilde{G}$. The closed
ear denoted by double dotted edges is replaced by the weighted edge
$v_{3}u$.}

\end{figure}

\begin{figure}
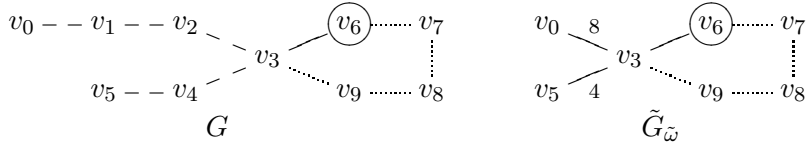

\begin{tabular}{ccc}
\input{simplifyExb.tex} &  & \input{simplifyEx2.tex}\tabularnewline
$G$ &  & $\tilde{G}_{\tilde{\omega}}$\tabularnewline
\end{tabular}

\caption{\label{cap:Simpli2}Simplification of $G$ with goal vertex $v_{6}$
such that $\pi(G,v_{6})=\pi(\tilde{G}_{\tilde{\omega}},v_{6})=22$.
The two cut ears denoted by dashed edges are replaced by the weighted
edges $v_{0}v_{3}$ and $v_{5}v_{3}$ in $\tilde{G}$. The open ear
denoted by dotted edges remains in $\tilde{G}$.}

\end{figure}

Figures~\ref{cap:Simpli1} and \ref{cap:Simpli2} show two examples
of simplified graphs using ears. The graph $G$ is the same in both
examples but the ears are different because the goal vertices are
different. The pebbling number of the graph is $\pi(G)=36$. The path
connecting $v_{3}$ to $v_{5}$ could have been simplified using leaves
in both examples. The path connecting $v_{0}$ to $v_{3}$ could have
been simplified using leaves as well, but only in the second example.
In both of these cases a further simplification is possible using
leaves and Proposition~\ref{pro:leaves1}.

Note that in the second example the end vertices $v_{3}$ and $v_{6}$
of the open ear are adjacent. This possibility is important to keep
in mind during the development of an algorithm to find open ears.

\section{Barely sufficient pebble distributions}

Let $\mathcal{D}(G)$ be the set of pebble distributions on the graph
$G$. For $p,q\in\mathcal{D}(G)$ we write $p\le q$ if $p(v)\le q(v)$
for all $v\in G$. This gives a partial order on $\mathcal{D}(G)$.
We write $p<q$ if $p\le q$ but $p\not=q$. It is clear that if a
goal vertex is reachable from $p$ and $p\le q$ then the goal vertex
is also reachable from $q$.
\begin{defn}
Let $x$ be a goal vertex of $G_{\omega}$. A pebble distribution
$p$ is \emph{sufficient} for $x$ if $x$ is reachable from $p$.
The set of sufficient distributions for $x$ is denoted by $\mathcal{S}(G_{\omega},x)$.
A pebble distribution $p\in\mathcal{S}(G_{\omega},x)$ is \emph{barely
sufficient} for $x$ if $x$ is not reachable from any pebble distribution
$q$ satisfying $q<p$. The set of barely sufficient distributions
for $x$ is denoted by $\mathcal{B}(G_{\omega},x)$. The set of \emph{insufficient
distributions} for $x$ is $\mathcal{I}(G_{\omega},x)=\mathcal{D}(G)\setminus\mathcal{S}(G_{\omega},x)$.
We are going to use the notation $\mathcal{S}(x)$, $\mathcal{B}(x)$
and $\mathcal{I}(x)$ if $G$ and $\omega$ is clear from the context.
\end{defn}
We can partition $\mathcal{B}(x)$ into the disjoint union $\mathcal{B}_{0}(x)\dot{\cup}\cdots\dot{\cup}\mathcal{B}_{k}(x)$
where $\mathcal{B}_{i}(x)$ contains those distributions in $\mathcal{B}(x)$
from which $x$ is reachable in $i$ pebbling moves but $x$ is not
reachable in fewer than $i$ moves. Note that the only element of
$\mathcal{B}_{0}(x)$ is the pebble distribution $1_{\{x\}}$ that
contains a single pebble on $x$.
\begin{example}
\label{exa:barely}Figure~\ref{cap:barelyExample} shows an example
of $\mathcal{B}(G,x)$.

\begin{figure}
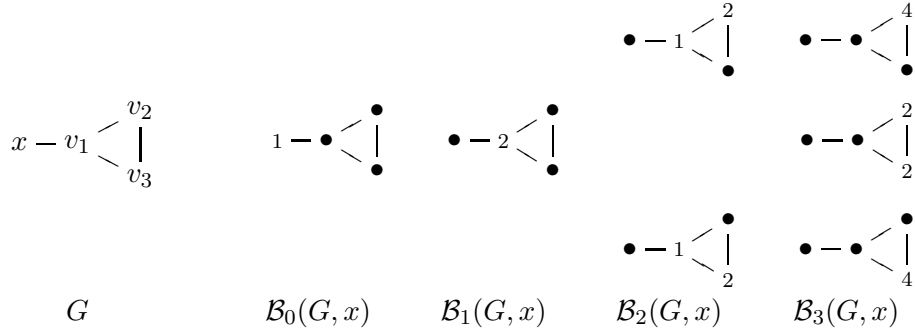

\begin{tabular}{cccccc}
 &  &  &  & \input{ex111.tex} & \input{ex1111.tex}\tabularnewline
\input{ex0.tex} &  & \input{ex1.tex} & \input{ex11.tex} &  & \input{ex1112.tex}\tabularnewline
 &  &  &  & \input{ex112.tex} & \input{ex1113.tex}\tabularnewline
$G$ & $\quad$ & $\mathcal{B}_{0}(G,x)$ & $\mathcal{B}_{1}(G,x)$ & $\mathcal{B}_{2}(G,x)$ & $\mathcal{B}_{3}(G,x)$\tabularnewline
\end{tabular}

\caption{\label{cap:barelyExample}The barely sufficient pebble distributions
for vertex $x$. The vertices denoted by bullets have no pebbles. }

\end{figure}

\end{example}
The following result is the main reason for our interest in barely
sufficient distributions.
\begin{prop}
\label{pro:finding_insuff}We have $p\in\mathcal{I}(G_{\omega},x)$
if and only if $q\le p$ does not hold for any $q\in\mathcal{B}(G_{\omega},x)$.\end{prop}
\begin{proof}
If $x$ is reachable from a pebble distribution $p$ then we can remove
pebbles from $p$ one by one if needed until we get a $q\in\mathcal{B}(G_{\omega},x)$
that satisfies $q\le p$. The other direction of the result is obviously
true.
\end{proof}
The following example shows how Proposition~\ref{pro:finding_insuff}
can be used to find the insufficient distributions.
\begin{example}
In Example~\ref{exa:barely} the maximal elements of $\mathcal{I}(G,x)$
are $p(x,v_{1},v_{2},v_{3})=(0,1,1,1)$, $q(x,v_{1},v_{2},v_{3})=(0,0,3,1)$
and $r(x,v_{1},v_{2},v_{3})=(0,0,1,3)$. The maximum size is $|q|=4=|r|$
and so $\pi(G,x)=5$. 
\end{example}
Our purpose now is to construct algorithms for finding $\mathcal{B}(G_{\omega},x)$
and $\pi(G_{\omega},x)$.

\section{Finding barely sufficient distributions}

The following result shows how a superset of $\mathcal{B}(G_{\omega},x)$
can be constructed using recursion starting at $\mathcal{B}_{0}(G_{\omega},x)=\{1_{\{x\}}\}$.

\begin{prop}
\label{pro:Binv}If $p\in\mathcal{B}_{i+1}(G_{\omega},x)$ then $p=q_{r^{-1}}$
for some $q\in\mathcal{B}_{i}(G_{\omega},x)$ and pebbling move $r$. \end{prop}
\begin{proof}
Suppose that $p\in\mathcal{B}_{i+1}(G_{\omega},x)$. Then there is
an executable sequence $s=(s_{1},\ldots,s_{i+1})$ of pebbling moves
such that $p_{s}(x)\ge1$. Then with $q=p_{s_{1}}$ we clearly have
$p=q_{s_{1}^{-1}}$. Vertex $x$ is reachable from $q$ in the $i$
moves of the sequence $(s_{2},\ldots,s_{i+1})$. If $x$ is reachable
from $q$ in $j$ moves then it is reachable from $p$ in $j+1$ moves.
So $x$ cannot be reached from $q$ in fewer than $i$ moves, which
means that $q\in\mathcal{B}_{i}(G_{\omega},x)$.
\end{proof}
We do not have to use every pebbling move $r$ during the construction
of $\mathcal{B}_{i+1}(G_{\omega},x)$ from $\mathcal{B}_{i}(G_{\omega},x)$
as shown in the next result that essentially a simple case of the
No-Cycle Lemma. 
\begin{prop}
\label{pro:initial}Let $p\in\mathcal{B}(G_{\omega},x)$. If $p_{S}\ge1_{\{x\}}$
and $(v\a u)\in S$ then $q=p_{(v\a u)^{-1}}\not\in\mathcal{B}(G_{\omega},x)$.\end{prop}
\begin{proof}
Let $\tilde{q}(u,v,*)=(q(u)-1,q(v)-1,q(*))$ and $R=S\setminus\{(v\a u)\}$.
Then $\tilde{q}<q$ and \begin{align*}
\tilde{q}_{R}(u,v,*) & =(\tilde{q}_{S}(u)-1,\tilde{q}_{S}(v)+2,\tilde{q}_{S}(*))\\
 & =(q_{S}(u)-2,q_{S}(v)+1,q_{S}(*))\\
 & =(p_{S}(u),p_{S}(v),p_{S}(*))=p_{S}(u,v,*)\end{align*}
which means $\tilde{q}_{R}=p_{S}\ge1_{\{x\}}$. So $q$ is not barely
sufficient.
\end{proof}
An important interpretation of this result is that every distribution
in $\mathcal{B}(G_{\omega},x)$ can be gotten as $p_{T}$ where $p=1_{\{x\}}$
and $T$ is a multiset of inverse pebbling moves such that $(v\a u)^{-1}$
and $(u\a v)^{-1}$ are not in $T$ together for any $u$ and $v$.
Keeping track of the directions of the inverse pebbling moves speeds
up the calculation of finding $\mathcal{B}(G_{\omega},x)$. It also
helps eliminating moves that cannot be initial moves. 

\incmargin{1em}
\begin{algorithm2e}[tp]
\dontprintsemicolon
\linesnumbered
\hrule
\medskip
\SetKwFunction{KwPB}{pushBack}
\SetKwFunction{KwInsert}{insert}
\SetKwFunction{KwRemove}{remove}
\SetKw{KwBreak}{break}
\SetKw{KwWhile}{while}
\KwIn{$G_\omega$, $x$}
\KwOut{$\mathcal{B}(G_\omega,x)$}
\smallskip
\hrule
\smallskip
$(p,E,W):=(1_{\{x\}}, E(\overrightarrow{G}),\emptyset) $ \hfill {\rm// (distribution, transfers, forbidden vertices)} \;
$Q$.\KwPB{$(p,E,W)$} \hfill {\rm// growing queue of distributions} \;
\For{$(p,E,W)\in Q$ } {
  \For(\hfill {\rm// $u$ has a pebble}){$u\in{\rm supp}(p)$}{ 
     \For(\hfill {\rm// allowed transfer from $v$ to $u$}){$(v,u)\in E$ {\rm and} $u\not\in W$   
       }{
      $q := p_{(v\a u)^{-1}}$ \hfill {\rm// candidate distribution} \;
      $F := E\setminus\{(u,v)\}$ \hfill {\rm// backward transfer no longer allowed} \;
      $X := W\cup \{u\}$ \hfill {\rm// transfer to $u$ no longer allowed} \;
      \For{$(\tilde{q},\tilde{F},\tilde{X})\in Q$}{
        \If(\hfill {\rm// candidate too large?}){$\tilde{q}<q$} { 
	  \KwBreak \hfill {\rm// candidate fails} \;
	}
        \If(\hfill {\rm// candidate already in queue?}){$\tilde{q}=q$ } {
	  $\tilde{F} := \tilde{F}\cap F$ \hfill {\rm// fewer allowed edges for $\tilde{q}$} \;
	  $\tilde{X} := \tilde{X}\cap X$ \hfill {\rm// not initial in any way} \;
	  \KwBreak \;
	}
        \If(\hfill {\rm// $\tilde{q}$ is not barely sufficient?}){$\tilde{q}>q$} { 
	  $Q$.\KwRemove{$(\tilde{q},\tilde{F},\tilde{X})$} \hfill {\rm// remove $\tilde{q}$ from queue } \;
	}	
      }    
      \If {\rm did not break} {
         $Q$.\KwPB{$(q,F,X)$} \hfill {\rm// candidate works} \;
      }
    }
  }
  \hfill {\rm// modification (see Note~\ref{not:modify-alg1})}
}
$\mathcal{B}(G_\omega,x):= \{ p \mid (p,E,W)\in Q\}$ \;
\medskip
\hrule
\caption{Algorithm to find the set $\mathcal{B}(G_\omega,x)$ of barely sufficient 
distributions.}
\label{alg1}
\end{algorithm2e}

Given a graph $G$, let $\overrightarrow{G}$ be the directed graph
whose vertex set is $V(G)$ and whose arrow set contains two arrows
$(u,v)$ and $(v,u)$ for every edge $uv\in E(G)$.
\begin{algorithm}
\label{alg:1}The algorithm shown in Figure~\ref{alg1} finds the
set of barely sufficient distributions.
\end{algorithm}
The heart of the algorithm is Proposition~\ref{pro:Binv}. We apply
inverse pebbling moves to transfer pebbles in barely sufficient distributions
in hope of finding new barely sufficient distributions. We use triples
of the form $(p,E,W)$ where $p$ is a pebble distribution. The role
of $E$ is to keep track of the direction of the pebble flow so that
we can avoid the back and forth transfer as explained in Proposition~\ref{pro:initial}.
The role of $W$ is to avoid pebbling sequences that are not regular
as explained in Theorem~\ref{thm:fundam}(4). Now we give the detailed
explanation of the algorithm:
\begin{itemize}
\item lines~1--2: We fill the queue $Q$ of barely sufficient distribution
candidates with $\mathcal{B}_{0}(G_{\omega},x)$. We set $E=E(\overrightarrow{G})$
since the pebbles can flow in any direction. We set $W=\emptyset$
since no vertex is ruled out as the starting vertex of an initial
move. 
\item line~3: This loop takes an element $p$ of $Q$ and applies a possible
inverse pebbling move to create a new distribution $q$. If $p\in\mathcal{B}_{i}(x)$
then $q$ is a candidate for $\mathcal{B}_{i+1}(x)$.
\item line~4: We find a vertex $u$ that has at least one pebble. We plan
to remove a pebble from this vertex and add two pebbles to an adjacent
vertex $v$.
\item line~5: We only want to apply $(v\a u)^{-1}$ if $(u\a v)^{-1}$
was not used before and if $(v\a u)$ is an initial move.
\item line~6: We apply the inverse pebbling move $(v\a u)^{-1}$ to create
the new barely sufficient candidate $q$.
\item line~7: According to Proposition~\ref{pro:initial}, we do not want
to apply $(u\a v)^{-1}$ since we already used $(v\a u)^{-1}$.
\item line~8: Any move of the form $(u\a w)$ is not an initial move since
we already have a move of the form $(v\a u)$.
\item line 9: The loop checks the newly created candidate against the other
distributions in the queue.
\item lines 10--11: We already put a smaller candidate is the queue so the
new candidate cannot be barely sufficient.
\item line 12: The new candidate $q$ is already in the queue. It is likely
that it was created using different inverse pebbling moves. We do
not add this candidate to the queue twice. Still, we can update the
information about this distribution in the queue. 
\item line 13: We can reduce the possible inverse pebbling moves using Proposition~\ref{pro:initial}.
\item line 14: It is possible that a move is initial in one set of pebbling
moves but not in another set. We only want to declare a move not initial
if it is not initial in every possible set of pebbling moves that
reaches the goal vertex. 
\item lines 16--17: If the new candidate is smaller than a distribution
$\tilde{q}$ in the queue then $\tilde{q}$ cannot be barely sufficient.
Therefore we remove it from the queue.
\item lines 18--19: The new candidate is added to the queue.
\end{itemize}
\begin{figure}
\input{barelyK3.tex}

\vskip .25in \hrule \vskip .1in 

\begin{tabular}{l|l}
1. $Q=\{p^{(0)}\}$ & 6. $p^{(31)}$ tested, smaller than $p^{(12)}$\tabularnewline
2. $Q=\{p^{(0)},p^{(11)}\}$ & 7. $Q=\{p^{(0)},p^{(11)},p^{(21)}\}$ \tabularnewline
3. $Q=\{p^{(0)},p^{(11)},p^{(12)}\}$ & 8. $Q=\{p^{(0)},p^{(11)},p^{(21)},p^{(31)}\}$\tabularnewline
4. $Q=\{p^{(0)},p^{(11)},p^{(12)},p^{(21)}\}$ & 9. $p^{(32)}$ tested but is larger than $p^{(0)}$\tabularnewline
5. $p^{(22)}$ tested but is larger than $p^{(21)}$ & 10. $p^{(41)}$ tested but is larger than $p^{(0)}$\tabularnewline
\end{tabular}

\vskip .1in\hrule

\caption{\label{cap:barelyK3}The distributions in solid frames belong to $\mathcal{B}(G_{\omega},x)$.
A distribution $p^{(ij)}$ in a dotted frame is never in the queue.
A solid arrow from $q$ to $p$ is drawn with label $(v\a u)$ if
$q=p_{(v\a u)^{-1}}$. A dashed arrow from $q$ to $p$ is drawn if
$q\ge p$ and so $q\not\in\mathcal{B}(G_{\omega},x)$. The table shows
how the distribution queue $Q$ changes during the execution of the
Algorithm~\ref{alg:1}.}

\end{figure}

\begin{example}
\label{exa:alg}Let $x$ be the goal vertex and $v_{1}$ and $v_{2}$
be the other vertices of the complete graph $G=K_{3}$ and let $\omega(xv_{2})=5$.
Figure~\ref{cap:barelyK3} shows how Algorithm~\ref{alg:1} finds
$\mathcal{B}(G_{\omega},x)=\{p^{(0)},p^{(11)},p^{(21)},p^{(31)}\}$.
Note that $p^{(12)}$ is added to the queue and only removed later
when $p^{(31)}$ is found. This late recognition of the fact that
$p^{(12)}$ is not barely sufficient is the reason why the algorithm
needs to test $p^{(22)}$ as a candidate.
\end{example}

\section{\label{sec:Squished}Squished distributions}

In this section we prove a version of the Squishing Lemma of \cite{Bunde_optimal}
using open ears. A pebble distribution is \emph{squished} on a path
$P$ if all the pebbles on $P$ are placed on a single vertex of $P$
or on two adjacent vertices of $P$. A pebble distribution can be
made squished on a path using squishing moves. 
\begin{lem}
(Squishing) If vertex $x$ is not reachable from a pebble distribution
$p$ with size $n$, then there is a pebble distribution of size $n$
that is squished on each unweighted open ear and from which $x$ is
still not reachable.\end{lem}
\begin{proof}
Let $E$ be an unweighted open ear with consecutive vertices $v_{0},\ldots,v_{n}$.
Suppose that the pebble distribution $p$ is not squished on $E$.
Let $i$ be the smallest and $j$ be the largest index for which $p(v_{i})>0$
and $p(v_{j})>0$. Note that we must have $j-i\ge2$. Define a new
pebble distribution $q$ by applying a squishing move that moves one
pebble from $v_{i}$ to $v_{k}$ and another pebble from $v_{j}$
to $v_{k}$ for some $k$ satisfying $i<k<j$. 

Suppose $x$ is reachable from $q$, that is, there is an acyclic
multiset $S$ of pebbling moves such that $q_{S}\ge1_{\{x\}}$. Pick
a maximal directed path of $T(G,S)$ with consecutive vertices $v_{k}=w_{0},w_{1},\ldots,w_{l}$
all in the set $\{v_{i},v_{i+1},\ldots,v_{j}\}$. Let $D$ be the
set of moves corresponding to the arrows of this directed path, that
is, $D=\{(w_{0}\a w_{1}),\ldots,(w_{l-1}\a w_{l})\}$ and let $R=S\setminus D$.
We need to consider three cases depending on whether $w_{l}=v_{k}$,
$w_{l}\in\{v_{i},v_{j}\}$ or $w_{l}\not\in\{v_{k},v_{i},v_{j}\}$.
It is easy to see that in all three cases we must have $p_{R}\ge1_{\{x\}}$
which is a contradiction. 

Applying squishing moves repeatedly on $E$ makes the pebble distribution
squished on $E$. This procedure keeps the goal vertex $x$ unreachable.
A squishing move on $E$ might remove a pebble from another open ear
but it cannot add a pebble to it. So if the distribution is squished
on an open ear then it remains squished after the application of a
squishing move on $E$. So the desired pebble distribution can be
reached by applying all the available squishing moves on all unweighted
ears in any order.
\end{proof}
The set $\mathcal{I}_{s}(G_{\omega},x)$ of \emph{squished insufficient
distributions} is the set of those elements of $\mathcal{I}(G_{\omega},x)$
that are squished on all open ears of $G$. The set $\mathcal{B}_{s}(G_{\omega},x)$
of \emph{squished barely sufficient distributions} is the set of those
elements of $\mathcal{B}(G_{\omega},x)$ that are squished on all
open ears of $G$.
\begin{prop}
\label{pro:finding_insuff_s}We have $p\in\mathcal{I}_{s}(G_{\omega},x)$
if and only if $q\le p$ does not hold for any $q\in\mathcal{B}_{s}(G_{\omega},x)$.\end{prop}
\begin{proof}
The result follows from Proposition~\ref{pro:finding_insuff}.\end{proof}
\begin{note}
\label{not:modify-alg1}We can find $\mathcal{B}_{s}(G_{\omega},x)$
by a slight modification of Algorithm~\ref{alg:1}. On line~20 we
simply remove $(p,E,W)$ from $Q$ if $p$ is not squished.\end{note}
\begin{cor}
\label{cor:pi-insuff-s}$\pi(G_{\omega},x)=\max\{|p|:p\in\mathcal{I}_{s}(G_{\omega},x)\}+1$. \end{cor}
\begin{proof}
The result follows from the Squishing Lemma using $\pi(G_{\omega},x)=\max\{|p|:p\in\mathcal{I}(G_{\omega},x)\}+1.$
\end{proof}

\section{\label{sec:find-insuf}Finding insufficient distributions}

\incmargin{1em}
\begin{algorithm2e}[tp]
\dontprintsemicolon
\linesnumbered
\hrule
\medskip
\SetKwFunction{KwPB}{pushBack}
\SetKwFunction{KwBP}{popBack}
\SetKwFunction{KwInsert}{insert}
\SetKwFunction{KwRemove}{remove}
\SetKw{KwBreak}{break}
\SetKw{KwContinue}{continue}
\SetKw{KwWhile}{while}
\KwIn{$\mathcal{C}:=\mathcal{B}_s(G_\omega,x)$}
\KwOut{$\pi(G_\omega,x)$}
\smallskip
\hrule
\smallskip
\For{$v\in V(G)$} {
  $\bar p(v):=\max\{q(v)\mid q\in\mathcal{C}\}$ \hfill 
  {\rm// $p$ is an upper bound for $\mathcal{I}_s(G_\omega,x)$}\;
}
$P$.\KwPB{$\bar p$} \hfill {\rm// ordered queue of not yet squished candidates}\;
\While(\hfill {\rm// more candidates to try}){$P$ {\rm not empty} } {
  $P$.\KwBP{p} \hfill {\rm// work with candidate $p$}\; 
  \If{p {\rm squished}} {
  $Q$.\KwInsert{$(p,1)$} \hfill {\rm// ordered queue of squished candidates} \; 
  \KwContinue \hfill {\rm// nothing more to do with $p$}\;    
  }
  \For(\hfill {\rm// try to improve $p$ }) {$v\in V(G)$} {
       $q:=p$ \hfill {\rm// modify $p$}\;
       $q(v):=0$ \hfill {\rm// this might make it squished} \;
       $P$.\KwInsert{$q$} \hfill {\rm// add improved candidate to the queue}\;
  }
}
$M:=0$ \hfill {\rm// size of best insufficient distribution so far} \;
\While(\hfill {\rm// more candidates to try}){$Q$ {\rm not empty} } {
  $Q$.\KwBP{$(p,i)$} \hfill {\rm// $p$ works for $\mathcal{C}[1],\ldots,\mathcal{C}[i-1]$} \;
  \If(\hfill {\rm// too few pebbles?}) {$|p|\le M$} {
    \KwContinue  \hfill {\rm// candidate has no hope to be better}
  }
  \While(\hfill {\rm// find first $i$ such that $\mathcal{C}[i]\le p$}  )
     {$i\le |\mathcal{C}|$ {\rm and} $\mathcal{C}[i]\not\le p$}{ 
     $i:=i+1$ \hfill {\rm// not found yet}\;
  }
  \If (\hfill {\rm// no such $i$, candidate works}) {$i>|\mathcal{C}|$} {
    $M:= |p|$ \hfill {\rm// $p$ is the best insufficient distribution so far}\;
    \KwContinue \hfill {\rm// nothing more to do with $p$}\;    
  }
    \For(\hfill {\rm// make candidate work for $\mathcal{C}[i]$}){$v\in V(G)$} {
       $q:=p$ \hfill {\rm// modify $p$}\;
       $q(v):=\mathcal{C}[i](v)-1$ \hfill {\rm// $q$ works for $\mathcal{C}[1],\ldots,\mathcal{C}[i]$} \;
       \If(\hfill{\rm// nonnegative number of pebbles on $v$?}){$q(v)\ge 0$} {
         $Q$.\KwInsert{$(q,i+1)$} \hfill {\rm// add improved candidate to the queue}\;
       }
    }
}
$\pi(G_\omega,x):=M+1$ \;
\medskip
\hrule
\caption{Algorithm to find the distribution with the most pebbles that is insufficient for the goal
vertex.}
\label{alg2}
\end{algorithm2e}

Now we present an algorithm for finding $\mathcal{I}_{s}(x)$. Let
$\bar{p}$ be the pebble distribution defined by $\bar{p}(v)=\max\{q(v)\mid q\in\mathcal{B}_{s}(x)\}$
for all $v\in V(G)$. It is clear that if $p\in\mathcal{I}_{s}(x)$
then $p\le m$. The idea of the algorithm is to decrease the number
of pebbles at certain vertices of $\bar{p}$ until it becomes insufficient. 
\begin{algorithm}
\label{alg:2}The algorithm shown in Figure~\ref{alg2} finds $\pi(G_{\omega},x)$
using $\mathcal{B}_{s}(G_{\omega},x)$.
\end{algorithm}
The algorithm uses Proposition~\ref{pro:finding_insuff_s} and Corollary~\ref{cor:pi-insuff-s}.
The input $\mathcal{B}_{s}(G_{\omega},x)$ is the output of the modified
Algorithm~\ref{alg:1} as explained in Note~\ref{not:modify-alg1}.
It contains all the squished barely sufficient distributions. Now
we give the detailed explanation of the algorithm:
\begin{itemize}
\item lines 1--3: For each vertex $v$ there is a barely sufficient distribution
that has pebbles only on $v$. So the distribution $\bar{p}$ is an
upper bound for $\mathcal{I}(G_{\omega},x)$. Every insufficient distribution
can be constructed form $\bar{p}$ by decreasing the number of pebbles
on some vertices.
\item lines 4--12: We only need to find those insufficient distributions
that are squished. So instead of using $\bar{p}$ we can use a new
distribution $p$ by removing all the pebbles from $\bar{p}$ at a
few vertices until the distribution becomes squished. Every squished
insufficient distribution can be constructed from one such $p$ by
decreasing the number of pebbles on some vertices. So we collect all
these $p$'s in the queue $Q$ on line 7. The queues $P$ and $Q$
are kept ordered so lines 7 and 12 can use binary searches to avoid
duplication in the queues.
\item line 13: The variable $M$ is initialized, it will contain the maximum
size $\max\{|p|:p\in\mathcal{I}_{s}(G_{\omega},x)\}$ of an insufficient
distribution. 
\item lines 14--27: We are finding $M$ using Corollaries~\ref{pro:finding_insuff_s}
and \ref{cor:pi-insuff-s}. Queue $Q$ contains pairs of the form
$(p,i)$. Such a $p$ is always smaller than the first $i-1$ elements
of $\mathcal{B}_{s}(G_{\omega},x)$. In each iteration we replace
$(p,i)\in Q$ by possibly several new elements in the queue of the
form $(q,i+1)$. Such a $q$ is now smaller than the first $i$ elements
of $\mathcal{B}_{s}(G_{\omega},x)$. We can find such a $q$ by decreasing
the number of pebbles at a vertex where the $i$-th element of $\mathcal{B}_{s}(G_{\omega},x)$
has pebbles. The loop in line 23 finds these vertices. Line 27 uses
binary search.
\item line 28: The pebbling number is calculated according to Corollary~\ref{cor:pi-insuff-s}.
\end{itemize}
~

It is important to keep our queues sorted and to use binary search
at the insert operations. Without this the algorithm becomes too slow
to be practical.

\section{Test results}

We tested our algorithms by calculating the pebbling number of every
connected graph with fewer than 10 vertices. We used Nauty \cite{nauty}
to generate these graphs and their automorphism groups. We simplified
each graph as follows. For each goal vertex we replaced every closed
ears and cut ears with a weighted edge as described in Subsections~\ref{sub:Closed-ears}
and \ref{sub:Cut-ears}. Then we recursively used available leaves
to simplify the graph as much as possible as described in Section~\ref{sec:Simp-leaves}.
Then we run Algorithms~\ref{alg:1} and \ref{alg:2} to find the
pebbling number of the simplified graph.

\begin{table}
\def\ns#1{\goodbreak\smallskip\noindent$|V(G)|=#1$\hfill\\
\smallskip}
\setlength{\columnseprule}{.5pt}\begin{multicols}{6}
\ns 1
\begin{verbatim}
1       1
\end{verbatim}

\hrule

\ns 2
\begin{verbatim}
2       1
\end{verbatim}

\hrule

\ns 3
\begin{verbatim}
3       1
4       1
\end{verbatim}

\hrule

\ns 4
\begin{verbatim}
4       3	
5       2

8       1	
\end{verbatim}

\hrule

\ns 5
\begin{verbatim}
5      10
6       5

8       2	
9       3

16      1	
\end{verbatim}

\hrule

\ns 6
\begin{verbatim}
6      45
7      15

8      13	
9      16
10     13
11      1

16      4	
17      4

32      1	
\end{verbatim}

\hrule

\ns 7
\begin{verbatim}
7     322
    
8     113
9     125
10    129
11     68
12      4

16     23	
17     35
18     22
19      2
      
32      4	
33      5
      
64      1	
\end{verbatim}

\hrule

\ns8	
\begin{verbatim}
8    4494
9    1658
10   1870
11   1425
12    478
13     26
14      1
      
16    190
17    341
18    333
19    148
20     15

32     36	
33     52
34     34
35      3

64      6	
65      6

128     1	
\end{verbatim}

\hrule

\ns 9	
\begin{verbatim}
9  126646
10  43935
11  41222
12  22756
13   4975
14    208
15      6

16   2505	
17   5293
18   5992
19   4070
20   1310
21    137
22      5
23      8

32    318	
33    626
34    579
35    261
36     33
39      1

64     50	
65     79
66     47
67      4

128     6	
129     7

256     1	
\end{verbatim}
\end{multicols}

\caption{\label{cap:spectrum}The frequency of pebbling numbers for graph with
less than 10 vertices. The data is grouped by the number $|V(G)|$
of vertices in the graph. Each data row contains a possible pebbling
number followed by the frequency of this pebbling number.}

\end{table}

The automorphism group helped us reducing the number of goal vertices
to representatives of orbits. It is well known that the hardest to
reach goal vertex in a tree is a leaf, so in trees we only picked
leaf vertices for the goal vertex. 

The algorithms were coded in C++ using the Standard Template Library.
The code was compiled with the gnu compiler. It took about a day on
a 3 GHz Unix machine to finish the calculations. The calculation for
graphs with fewer than 9 vertices took less than 10 minutes. Table~\ref{cap:spectrum}
shows the frequency of the pebbling numbers. The result confirms the
existence of gaps in the spectrum of pebbling numbers described in
\cite{Cabanski}. We checked our results on many graphs with known
pebbling numbers such as paths, complete graphs, cycles, some trees,
Lemke graph, Petersen graph. The pebbling numbers are available on
the internet.

\section{Further questions}
\begin{enumerate}
\item We do not have any example where $t\mapsto\pi_{t}(G,x)$ is not linear
for $t\ge3$. Is it true that $t\mapsto\pi_{t}(G,x)$ is always linear
for $t\ge t_{0}$ for some $t_{0}$? What properties of the graph
can be used to find $t_{0}$? 
\item Is it possible that the goal vertex $x$ that maximizes $\pi(G,x)$
is an interior vertex of a cut ear of $G$? The answer seems to be
no and it may depend on the first question. We could use this result
to speed up our calculation of $\pi(G)$ since we would have to test
fewer goal vertices. 
\item Is it possible to take advantage of open ears in a better way? Can
we simplify a graph with an open ear so that it has fewer vertices?
This simplification might help tremendously if it is compatible with
ear decomposition. It might be possible to reduce a graph completely
if we know $\pi_{t}$ for all graphs with fewer vertices.
\item Perhaps adding extra weighted edges could speed up Algorithm~\ref{alg:1}
in certain cases. For example in an open ear, we could connect the
end vertices to the interior vertices with appropriate weights depending
on the distance.
\item What general results are there about the pebbling number of weighted
graphs? In particular, does Graham's conjecture hold for weighted
graphs? 
\item What is the pebbling number of simple weighted graphs like the weighted
complete graph or a weighted cycle? Even a weighted triangle seems
to be fairly complicated with a lot of cases to consider.
\end{enumerate}
\bibliographystyle{amsplain}
\bibliography{pebble}

\providecommand{\bysame}{\leavevmode\hbox to3em{\hrulefill}\thinspace}
\providecommand{\MR}{\relax\ifhmode\unskip\space\fi MR }
\providecommand{\MRhref}[2]{%
  \href{http://www.ams.org/mathscinet-getitem?mr=#1}{#2}
}
\providecommand{\href}[2]{#2}
\begin{thebibliography}{1}

\bibitem{Bunde_optimal}
David~P. Bunde, Erin~W. Chambers, Daniel Cranston, Kevin Milans, and Douglas~B.
  West, \emph{Pebbling and optimal pebbling in graphs}, J. Graph Theory
  \textbf{57} (2008), no.~3, 215--238.

\bibitem{Cabanski}
Christopher Cabanski, \emph{Forbidden pebbling numbers}, University of Dayton
  honors thesis, 2007.

\bibitem{Betsy}
Betsy Crull, Tammy Cundiff, Paul Feltman, Glenn~H. Hurlbert, Lara Pudwell,
  Zsuzsanna Szaniszlo, and Zsolt Tuza, \emph{The cover pebbling number of
  graphs}, Discrete Math. \textbf{296} (2005), no.~1, 15--23.

\bibitem{Hurlbert_survey1}
Glenn Hurlbert, \emph{A survey of graph pebbling}, Proceedings of the Thirtieth
  Southeastern International Conference on Combinatorics, Graph Theory, and
  Computing (Boca Raton, FL, 1999), vol. 139, 1999, pp.~41--64.

\bibitem{Hurlbert_survey2}
\bysame, \emph{Recent progress in graph pebbling}, Graph Theory Notes N. Y.
  \textbf{49} (2005), 25--37.

\bibitem{tpebbling}
A.~Lourdusamy and S.~Somasundaram, \emph{The {$t$}-pebbling number of graphs},
  Southeast Asian Bull. Math. \textbf{30} (2006), no.~5, 907--914.

\bibitem{nauty}
Brendan~D. McKay, \emph{Practical graph isomorphism}, Proceedings of the Tenth
  Manitoba Conference on Numerical Mathematics and Computing, Vol. I (Winnipeg,
  Man., 1980), vol.~30, 1981, pp.~45--87.

\bibitem{Milans}
Kevin Milans and Bryan Clark, \emph{The complexity of graph pebbling}, SIAM J.
  Discrete Math. \textbf{20} (2006), no.~3, 769--798 (electronic).

\bibitem{Moews}
David Moews, \emph{Pebbling graphs}, J. Combin. Theory Ser. B \textbf{55}
  (1992), no.~2, 244--252.

\end{thebibliography}

\end{document}